\documentclass[12pt,oneside,english]{amsart}

\usepackage[T1]{fontenc}
\usepackage[latin9]{inputenc}
\usepackage{geometry}
\usepackage{babel}
\usepackage{amsbsy}
\usepackage{amstext}
\usepackage{amsthm}
\usepackage{amssymb}

 \usepackage{amsaddr}

\usepackage{xcolor}

\usepackage{graphicx}

\usepackage[unicode=true,pdfusetitle,
 bookmarks=true,bookmarksnumbered=false,bookmarksopen=false,
 breaklinks=false,pdfborder={0 0 1},backref=false,colorlinks=false]
 {hyperref}

\makeatletter

\numberwithin{figure}{section}

\theoremstyle{plain}
\newtheorem{thm}{\protect\theoremname}[section]
\newtheorem{conjecture}[thm]{\protect\conjecturename}
\theoremstyle{definition}
\newtheorem{defn}[thm]{\protect\definitionname}
\newtheorem{example}[thm]{\protect\examplename}
\theoremstyle{plain}
\newtheorem{lem}[thm]{\protect\lemmaname}
\theoremstyle{definition}
\newtheorem{problem}[thm]{\protect\problemname}

\makeatother

\providecommand{\definitionname}{Definition}
\providecommand{\examplename}{Example}
\providecommand{\lemmaname}{Lemma}
\providecommand{\theoremname}{Theorem}
\providecommand{\conjecturename}{Conjecture}
\providecommand{\problemname}{Problem}

\begin{document}

\title{Even spheres as joint spectra of matrix models}

\author{Alexander Cerjan}
\email{awcerja@sandia.gov}
\address{Center for Integrated Nanotechnologies, Sandia National Laboratories, Albuquerque, New Mexico 87185, USA}

\author{Terry A.\ Loring}
\email{tloring@unm.edu}
\address{Department of Mathematics and Statistics, University of New Mexico, Albuquerque, New Mexico 87131, USA}

\date{\today}


\begin{abstract}
The Clifford spectrum is a form of joint spectrum for noncommuting
matrices. This theory has been applied in photonics, condensed
matter and string theory.  In applications, the Clifford spectrum can be efficiently approximated using numerical methods, but this only is possible in low dimensional example.
Here we examine the higher-dimensional spheres that can arise from
theoretical examples.   We also describe a constuctive method to generate  five real symmetric almost commuting matrices that have a $K$-theoretical obstruction to being close to commuting matrices.  For this, we look to matrix
models of topological electric circuits.
\end{abstract}

\maketitle

\section{Multivariable spectrum}

There are many in-equivalent ways to define a joint spectrum of a
$d$-tuple $(A_{1},\dots,A_{d})$ of Hermitian $n$-by-$n$ matrices. 
For example, there is the monogenic spectrum \cite{Jefferies_McIn_Weyl_Calc}
which has a close connection with a noncommutative functional calculus.
Here we are concerned with the Clifford spectrum, as defined in \cite{Kisil_monogenic_func_calc},
which is useful in many parts of physics such as high-energy physics
\cite{berenstein2012matrix_embeddings,deBadyn_Karczmarek2015emergent_geometry},
condensed matter physics \cite{cheng2023revealing,franca2023obstructions,liu2018KitaevChainsJosephJunct,schuba2022Localizer_semimetals} and
photonic crystals \cite{CerjanLoring2022TopoPhtonics,dixon2023PhotonicHeterostrcutre}. These papers
make use of the connection between the Clifford spectrum and $K$-theory
\cite{Doll_schuba_Z2_flows_skew_localizer,LoringSchuba_even_dim_localizer}.
Here, we begin to investigate the question of what higher dimensional spaces
can occur as the Clifford spectrum of a $d$-tuple $(A_{1},\dots,A_{d})$
of Hermitian matrices and calculate the $K$-theory associated with these spaces. 

We make extensive use of the complex Clifford algebras $\mathcal{C}\ell(d)$.
We tend to think of $\mathcal{C}\ell(n)$ as the universal unital
$C^{*}$-algebra for generators $e_{1},\dots,e_{d}$ subject to relations
\begin{align*}
e_{j}^{*} & =e_{j},\quad(j=1,\dots,d)\\
e_{j}^{2} & =1,\quad(j=1,\dots,d)\\
e_{j}e_{k} & =-e_{k}e_{j},\quad(j\neq k)
\end{align*}
and so refer to $(e_{1},\dots,e_{d})$ as the \emph{universal Clifford generators}.
We also consider finite matrices $(\gamma_{1},\dots,\gamma_{d})$
that satisfy the above relations, and call such a $d$-tuple a \emph{representation of
the Clifford relations}. 

\begin{defn}
Suppose $A_{1},\dots,A_{d}$ are Hermitian matrices, all in $\boldsymbol{M}_{n}(\mathbb{C})$.
The \emph{Clifford spectrum} $\Lambda(A_{1},\dots,A_{d})$ is the set of
$\boldsymbol{\lambda}$ in $\mathbb{R}^{d}$ such that $L_{\boldsymbol{\lambda}}(A_{1},\dots,A_{d})$
is a non-invertible element of $\boldsymbol{M}_{n}(\mathbb{C})\otimes\mathcal{C}\ell(d)$,
where
\begin{equation*}
L_{\boldsymbol{\lambda}}(A_{1},\dots,A_{d})=\sum \left(A_{j} - \lambda_j\right)\otimes e_{j}.
\end{equation*}
\end{defn}

The element $L_{\boldsymbol{\lambda}}(A_{1},\dots,A_{d})$ is referred
to as the \emph{spectral localizer} in physics \cite{CerjanLoring2022TopoPhtonics,LoringSchuba_even_dim_localizer}.
Kisil takes closure when defining the Clifford spectrum, but we find
this is not needed as $\Lambda(A_{1},\dots,A_{d})$ as defined above
is automatically closed. This is because it is the zero-set of the
scalar valued function
\begin{equation*}
\mu_{\boldsymbol{\lambda}}^{C}(A_{1},\dots,A_{d})=s_{\min}(L_{\boldsymbol{\lambda}}(A_{1},\dots,A_{d}))
\end{equation*}
where we use $s_{\min}$ to indicate the smallest singular value of
a matrix. The function $\boldsymbol{\lambda}\mapsto\mu_{\boldsymbol{\lambda}}^{C}(A_{1},\dots,A_{d})$
is the Clifford pseudospectrum introduced in \cite{loring2015Pseudospectra_topo_ins}.
This function is continuous, even Lipschitz \cite{loring2015Pseudospectra_topo_ins},
so its zero-set is already closed.

The term Clifford spectrum is used in various incompatible ways in the mathematics literature \cite{Arveson_Dirac_operator,KisilRamirex_Cliff_func_calc}.  There seems to be no single notion of joint spectrum that is best for non-commuting matrices.  What we are calling the Clifford spectrum behaves in mathematically odd ways with respect to functional calculus, but it links very well to topological invariants and bound states in physics.

For many reasons, including minimizing computer memory use, we want to represent the $e_j$ by the smallest matrices possible.  It turns out that having a faithful representation of the complex Clifford algebra is not needed.  In fact, we can use any nontrival representation, with irreducible representations generally preferred.

\begin{lem}
Suppose $\gamma_{1},\dots,\gamma_{d}$ are in $\boldsymbol{M}_{r}(\mathbb{C})$
with $r>0$ and these form a representation of the Clifford relations.
If $A_{1},\dots,A_{d}$ are Hermitian matrices then
\begin{equation*}
s_{\min}\left(\sum_{j=1}^{d}\left(A_{j}-\lambda_{j}\right)\otimes\gamma_{j}\right)=s_{\min}\left(\sum_{j=1}^{d}\left(A_{j}-\lambda_{j}\right)\otimes e_{j}\right).
\end{equation*}
\end{lem}

\begin{proof}
When $d$ is even, $\mathcal{C}\ell(d)$ is isomorphic to $\boldsymbol{M}_{2^{d/2}}(\mathbb{C})$.
Up to unitary equivalence, the only option for the $\gamma_{j}$ is  
a direct sum of copies of the $e_{j}$. This means that, if we ignore multiplicity,
the spectrum of $\sum\left(A_{j}-\lambda_{j}\right)\otimes\gamma_{j}$
and $\sum\left(A_{j}-\lambda_{j}\right)\otimes e_{j}$ will be the
same.

When $d$ is odd, $\mathcal{C}\ell(d)$ is isomorphic to $\boldsymbol{M}_{m}(\mathbb{C})\oplus\boldsymbol{M}_{m}(\mathbb{C})$
with $m=2^{(d-1)/2}$. We have in this case two fundamental representations
of the Clifford relations, 
$
\alpha_{1},\dots,\alpha_{d}
$
and
$
-\alpha_{1},\dots,-\alpha_{d}
$
in $\text{M}_{m}(\mathbb{C})$. There is a unitary $Q$ so that 
\begin{equation*}
\gamma_{j}=Q\left[\begin{array}{cc}
\alpha_{j}\otimes I_{p} & 0\\
0 & -\alpha_{j}\otimes I_{q}
\end{array}\right]Q^{*}
\end{equation*}
where at least one of $p$ and $q$ is positive since $p+q=r$. In
this case, the spectrum of $\sum\left(A_{j}-\lambda_{j}\right)\otimes\gamma_{j}$
and $\sum\left(A_{j}-\lambda_{j}\right)\otimes e_{j}$ may differ,
but these will have the same singular values since 
\begin{equation*}
\sum\left(A_{j}-\lambda_{j}\right)\otimes(-\alpha_{j})=-\sum\left(A_{j}-\lambda_{j}\right)\otimes\alpha_{j}.
\end{equation*}
\end{proof}

There are precious few noncommutative examples with $d\geq 2$ where we can exactly calculate, by hand, the Clifford spectrum.   When $d=2$ the Clifford spectrum of $(A_1,A_2)$ is basically the same as the ordinary spectrum of $A_1 + iA_2$, so this case is understood.  More precisely, one  can prove  \cite{DeBonisLorSver_joint_spectrum} that $(x,y)$ is in the Clifford spectrum of the pair exactly when $x+iy$ is in the ordinary spectrum of $A_1 + i A_2$. 
Kisil \cite{Kisil_monogenic_func_calc} finds an example of three 2-by-2 matrices whose Clifford spectrum is a 2-sphere.  See
 \cite[\S IV]{berenstein2012matrix_embeddings} for a proof that the three matrices  that generate a fuzzy sphere also have Clifford spectrum a 2-sphere.  In \cite[\S 3.1]{sykora2016fuzzy_space_kit} more examples of three 2-by-2 matrices are examined, where the Clifford spectrum can be one sphere or two spheres, possibly touching.  Utilizing computer algebra systems one can get more examples \cite{DeBonisLorSver_joint_spectrum,sykora2016fuzzy_space_kit}, but these are still limited to examining only relatively small $d$-tuples of small matrices.
These examples showed it possible to have the Clifford spectrum for four Hermitian matrices to be homeomorphic to a torus or a three-sphere.

In many applications, for example in \cite{CerjanLoring2022TopoPhtonics,franca2023obstructions,liu2018KitaevChainsJosephJunct},
it suffices to use numerical computer methods to estimate the function 
$\boldsymbol{\lambda}\mapsto\mu_{\boldsymbol{\lambda}}^{C}(A_{1},\dots,A_{d})$.
Since small perturbations of a function can drastically change its
zero-set, this does not really help understand the structure of the
Clifford spectrum.

In this work, by systematically exploring what can lead to a symmetry in the Clifford
spectrum, we are able to make explicit calculations in examples where the Clifford
spectrum is any even-dimensional sphere. Later in the paper we
resort to numerical calculations inspired from physics to find
an explicit example of five almost commuting real symmetric matrices
that have a $K$-theoretical obstruction keeping them far from commuting
real symmetric matrices. This somewhat settles the mystery raised in
\cite{BoersLorRuiz_Pictures_K-theory}, where $KK$-theory and $E$-theory were used to show that such matrices must exist,
but with no hint of how to find these matrices.

\section{Symmetries in the Clifford Spectrum \label{sec:symmetries}}

Suppose $\gamma_{1},\dots,\gamma_{d}$ form a representation of the
Clifford relations. If $U=[u_{ij}]\in O(d)$ is a real orthogonal
matrix we get another representation of the Clifford relations by
defining 
\begin{equation*}
\hat{\gamma}_{j}=\sum_{r=1}^{d}u_{jr}\gamma_{r}.
\end{equation*}
This claim is easy to verify. The main calculation needed is
\begin{equation*}
\hat{\gamma}_{j}\hat{\gamma}_{k}=\left(\sum_{r}u_{jr}u_{kr}\right)I+\sum_{r<s}\left(u_{jr}u_{ks}-u_{kr}u_{js}\right)\gamma_{r}\gamma_{s}.
\end{equation*}
Setting $j=k$ we find
\begin{equation*}
\hat{\gamma}_{j}^{2}=\left(\sum_{r}u_{jr}u_{jr}\right)I=I
\end{equation*}
 and for $j\neq k$, 
\begin{equation*}
\hat{\gamma}_{j}\hat{\gamma}_{k}=\sum_{r<s}\left(u_{jr}u_{ks}-u_{kr}u_{js}\right)\gamma_{r}\gamma_{s}
\end{equation*}
which implies $\hat{\gamma}_{j}\hat{\gamma}_{k}=-\hat{\gamma}_{k}\hat{\gamma}_{j}$.

\begin{lem}
\label{lem:Unitary_action_on_spectrum} Suppose $\left(A_{1},\dots,A_{d}\right)$
is a $d$-tuple of Hermitian matrices in \textbf{$\boldsymbol{M}_{n}(\mathbb{C})$}
and that $U\in O(d)$. Suppose $\boldsymbol{\mathbb{\lambda}}\in\mathbb{R}^{d}$.
The $d$ matrices
\begin{equation*}
\hat{A}_{j}=\sum_{s}u_{js}A_{s}
\end{equation*}
are also Hermitian and 
\begin{equation*}
\boldsymbol{\mathbb{\lambda}}\in\Lambda\left(A_{1},\dots,A_{d}\right)\iff U\boldsymbol{\mathbb{\lambda}}\in\Lambda\left(\hat{A}_{1},\dots,\hat{A}_{d}\right).
\end{equation*}
\end{lem}

\begin{proof}
Select any $\gamma_{1},\dots,\gamma_{d}$ that form a representation
of the Clifford relations. Since $U^{\top}$ is just as orthogonal
as $U$ we know that the matrices
\begin{equation*}
\tilde{\gamma}_{j}=\sum u_{rj}\gamma_{r}
\end{equation*}
also form a representation of the Clifford relations. We can compute
the Clifford spectrum using the $\tilde{\gamma}_{j}$ and so look at
\begin{equation*}
\sum_{j}(A_{j}-\lambda_{j}I)\otimes\tilde{\gamma}_{j}=\sum_{j}A_{j}\otimes\tilde{\gamma}_{j}-\sum_{j}\lambda_{j}I\otimes\tilde{\gamma}_{j}.
\end{equation*}
We find
\begin{align*}
\sum_{j}A_{j}\otimes\tilde{\gamma}_{j} & =\sum_{j}\sum_{r}A_{j}\otimes u_{rj}\gamma_{r}\\
 & =\sum_{r}\sum_{j}u_{rj}A_{j}\otimes\gamma_{r}\\
 & =\sum_{j}\sum_{r}u_{jr}A_{r}\otimes\gamma_{j}\\
 & =\sum_{j}\hat{A}_{j}\otimes\gamma_{j}.
\end{align*}
Substituting $A_{j}$ by $\lambda_{j}I$ we find 
\begin{equation*}
\sum_{j}\lambda_{j}I\otimes\tilde{\gamma}_{j}=\sum_{j}\alpha_{j}I\otimes\gamma_{j}
\end{equation*}
where $\boldsymbol{\alpha}=U\boldsymbol{\mathbb{\lambda}}$. Thus 
\begin{equation*}
\sum_{j}(A_{j}-\lambda_{j}I)\otimes\tilde{\gamma}_{j}=\sum_{j}(\hat{A}_{j}-\alpha_{j}I)\otimes\gamma_{j}
\end{equation*}
and we are done.
\end{proof}

\begin{thm}
\label{thm:symmetry_in_spectrum}
Suppose $\left(A_{1},\dots,A_{d}\right)$ is a $d$-tuple of Hermitian
matrices in \textbf{$\boldsymbol{M}_{n}(\mathbb{C})$} and that $U\in O(d)$.
Let 
\begin{equation*}
\hat{A}_{j}=\sum_{s}u_{js}A_{s}.
\end{equation*}
If there exists $Q\in U(n)$ such that
\begin{equation*}
Q\hat{A_{j}}Q^{*}=A_{j}
\end{equation*}
for all $j$ then 
\begin{equation*}
\boldsymbol{\mathbb{\lambda}}\in\Lambda\left(A_{1},\dots,A_{d}\right)\iff U\boldsymbol{\mathbb{\lambda}}\in\Lambda\left(A_{1},\dots,A_{d}\right).
\end{equation*}
\end{thm}

\begin{proof}
Since unitarily equivalent $d$-tuples have the same Clifford spectrum, this follows from Lemma~\ref{lem:Unitary_action_on_spectrum}.
\end{proof}

\section{Scaling alters connectivity}

We regard the lack of a spectral mapping theorem to be a feature,
not a bug, when it comes to applications in physics. We first present a
simple example where applying an affine function to a triple of matrices
does not correspond to an affine transformation applied to the Clifford
spectrum.  The transformation we apply is simply rescaling in two of the dimensions.  Depending on the
the size of the rescaling, the Clifford spectrum remains sphere-like or breaks into pieces.  This is
reminiscent of the transition from individual atoms to a molecule. To demonstrate this kind of transition in a physical system explicitly, we then consider a second example rooted in a well-known 2D lattice whose band structure can be tuned to possess non-zero first Chern numbers \cite{haldane_model_1988}.

\begin{figure}
\includegraphics{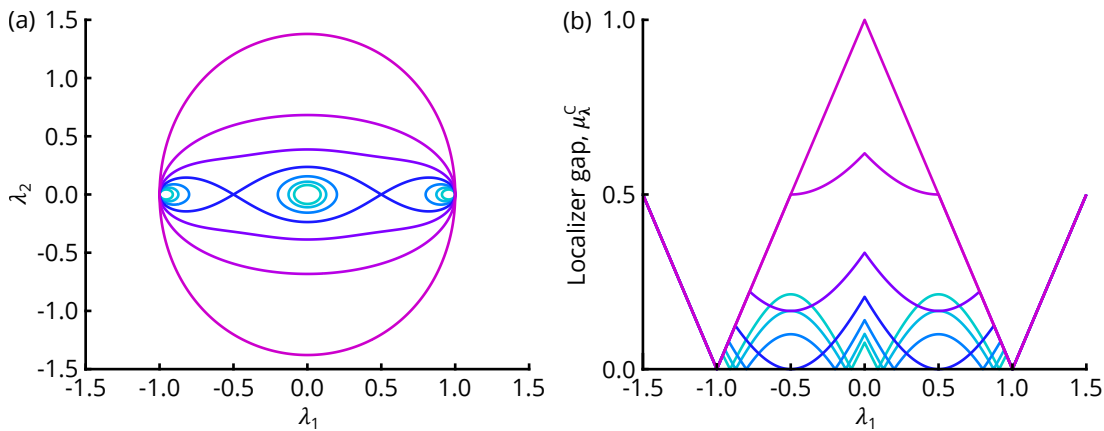}
\caption{(a) The Clifford spectrum $\Lambda(A,tB,tC)$ for $t=1/7,1/6,\dots,1$, for fixed $\lambda_3 = 0$. (b) Plot of the localizer gap as $\lambda_1$ varies for the same set of $t$, with $\lambda_2 = \lambda_3 = 0$. The coloration is consistent between both plots, with $t = 1/7$ corresponding to teal and $t = 1$ displayed as magenta. \label{fig:affine_example}}
\end{figure}

\begin{example}
Consider three matrices,
\[
A=\left[\begin{array}{ccc}
-1 & 0 & 0\\
0 & 0 & 0\\
0 & 0 & 1
\end{array}\right],\ B=\left[\begin{array}{ccc}
0 & 1 & 0\\
1 & 0 & 1\\
0 & 1 & 0
\end{array}\right],\ C=\left[\begin{array}{ccc}
0 & i & 0\\
-i & 0 & i\\
0 & -i & 0
\end{array}\right].
\]
We look at the Clifford spectrum of $\left(A,tB,tC\right)$ for various
values of $t$ between $1/7$ and $1$. We have rotational symmetry
in the second and third coordinate so it suffices to display the slice
with $\lambda_{3}=0$, as in Fig.~\ref{fig:affine_example}. There is a transition in the topology of the
spectrum at $t=1/4$, as smaller positive values of $t$ lead to three
separated surfaces each of which are homeomorphic to a sphere, while larger values of $t$ lead
to a (single) connected surface. This example is similar to
\cite[Example 4.4]{DeBonisLorSver_joint_spectrum}.

\end{example}

The transition observed in the Clifford spectrum, from being a collection of disconnected spheroids to being a single spheroid, manifests in some crystalline materials when the spacing between the constituent atoms or molecules is changed relative to the system's energy scale, in some dimensionless sense. When the spacing between these elements is sufficiently small so that the system behaves as a crystal, the Clifford spectrum is a connected surface. When the spacing between the elements increases beyond some critical value, this surface breaks apart into many separated surfaces, and the system behaves as though it is a collection of decoupled elements.

\begin{example}
To demonstrate this transition, consider a finite piece of a Haldane lattice \cite{haldane_model_1988}, which is a honeycomb lattice that contains both nearest-neighbor and next-nearest-neighbor couplings. The Hamiltonian for this lattice can be written in a tight-binding basis as
\begin{align}
    H =& M \sum_{m,n} \left(a_{m,n}^\dagger a_{m,n} - b_{m,n}^\dagger b_{m,n} \right) - t \sum_{\langle (m,n),(m',n')\rangle} \left(b_{m',n'}^\dagger a_{m,n} + a_{m,n}^\dagger b_{m',n'} \right) \notag \\
    & -t_\textrm{c} \sum_{\langle \langle (m,n),(m',n')\rangle \rangle} \left(e^{i \phi} a_{m',n'}^\dagger a_{m,n} + e^{-i \phi} a_{m,n}^\dagger a_{m',n'} + e^{i \phi} b_{m',n'}^\dagger b_{m,n} + e^{-i \phi} b_{m,n}^\dagger b_{m',n'} \right). \label{eq:Hhal}
\end{align}
Here, $a_{m,n}$ and $b_{m,n}$ ($a_{m,n}^\dagger$ and $b_{m,n}^\dagger$) are the annihilation (creation) operators on the two constituent sublattices in the unit cell identified by the index $(m,n)$, and we are using notation that is common in the physics literature for these types of systems with $c^\dagger$ denoting the conjugate transpose of $c$. The two sublattices have opposite on-site energies (i.e., diagonal elements) $\pm M$. The nearest-neighbor couplings have strength $t$, and the summation $\langle (m,n),(m',n')\rangle$ only includes those lattice sites in the same or adjacent unit cells that are nearest neighbors. The next-nearest-neighbor couplings have strength $t_{\textrm{c}}$ and a direction dependent phase $\phi$; the summation over $\langle \langle (m,n),(m',n')\rangle \rangle$ denotes these next-nearest-neighbor pairs of lattice sites. 

The other two matrices that are combined to form this system's spectral localizer so as to calculate its Clifford spectrum are its position matrices. In the tight-binding basis, these matrices are diagonal, and $X$ and $Y$ contain the coordinates of each lattice site $(x_m, y_n)$ in both sublattices. For the plot shown in Fig.~\ref{fig:kappa}, the site-to-site spacing is $a$, such that the crystal's lattice vectors have length $\sqrt{3}a$. Altogether, the spectral localizer for the Haldane lattice can be written (using the Pauli spin matrices as the Clifford representation) as
\begin{equation}
    L_{\boldsymbol{\lambda}=(\kappa_X x,\kappa_X y,\kappa_H E)}(\kappa_X X,\kappa_X Y,\kappa_H H)= \kappa_X \left(X - x\right)\otimes \sigma_x + \kappa_X \left(Y - y\right)\otimes \sigma_y + \kappa_H \left(H - E\right) \otimes \sigma_z. \label{eq:2DL}
\end{equation}
Here, we have included two dimension-full scaling coefficients, $\kappa_X$ and $\kappa_H$ that have units of inverse distance and inverse energy, respectively, such that $L_{\boldsymbol{\lambda}}$ is dimensionless. Moreover, we are directly identifying the components of $\boldsymbol{\lambda}=(\kappa_X x,\kappa_X y,\kappa_H E)$ as the corresponding physically meaningful values of position $(x,y)$ and energy $E$, but similarly re-scaled to be dimensionless.

Beyond enforcing consistent units in the spectral localizer, the two scaling coefficients $\kappa_X$ and $\kappa_H$ can be heuristically thought of as adjusting the weights given to the position matrices $X -x$ and $Y-y$ relative to lattice's Hamiltonian $H - E$ in the Clifford spectrum $\Lambda(\kappa_X X,\kappa_X Y,\kappa_H H)$. As these weights are adjusted, the Clifford spectrum is either a connected surface, when the lattice's Hamiltonian is more heavily weighted, or many disconnected sphereoids, when the lattice's position operators are more heavily weighted. In the case of the latter, the Clifford spectrum is responding directly to the locations of the individual lattice sites in the system, yielding these many disconnected surfaces centered near each site. In contrast, when the position operators are de-emphasized, the Clifford spectrum reveals information about the lattice as a whole, with the single spheroid in the Clifford spectrum being associated with this lattice's well-known chiral edge states \cite{cerjan_quadratic_2022}.
\end{example}

\begin{figure*}
\includegraphics{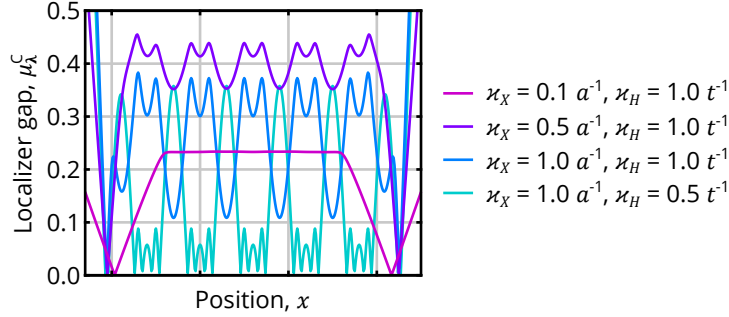}
\caption{Localizer gap $\mu_{\boldsymbol{\lambda}}^{\textrm{C}}$ for a 12-by-12 Haldane lattice whose Hamiltonian is given by Eq.\ (\ref{eq:Hhal}), with $t_{\textrm{c}} = 0.5 t$, $\phi = \pi /6$ and $M = 0$. Here, only $x$ is varied, with $y$ fixed at one of the rows of lattice sites and $E = 0$. Different colors show different values of $\kappa_{X}$ and $\kappa_{H}$, the scaling coefficients in Eq.\ (\ref{eq:2DL}.}
\label{fig:kappa}
\end{figure*}

\section{The $K$-theory of the Clifford resolvent set }

In many examples with $d\geq3$ the Clifford spectrum of $\left(A_{1},\dots,A_{d}\right)$
is a $(d-1)$-dimensional surface in $\mathbb{R}^{d}$. It is frequently
possible to associate a $K$-theory element to the connected components
of the complement of the Clifford spectrum. This is typically in $KO_{j}(\mathbb{R})$
and so is often computed as an integer or element of $\mathbb{Z}/2\mathbb{Z}$. 
Here, we will only need integer invariants. These invariants are (up to
an isomorphism) computed as some multiple of
\begin{equation*}
\mathrm{sig}\left(L_{\boldsymbol{\lambda}}(A_{1},\dots,A_{d})\right)
\end{equation*}
where we take $\textrm{sig}(L)$ to indicate signature, which for
a Hermitian, invertible matrix $L$ is the difference between the
number of positive and of negative eigenvalues. Here we do need to
specify a particular choice of the $\gamma_{j}$. These need to form
an irreducible representation of the Clifford relations. When $d$
is odd we need to make an arbitrary choice of which irreducible representation
to use. The other choice just flips the sign of the signature.

We know \cite{loring2015Pseudospectra_topo_ins} that $\mathrm{sig}\left(L_{\boldsymbol{\lambda}}(A_{1},\dots,A_{d})\right)=0$
when $|\boldsymbol{\lambda}|$ is large or when the $A_{j}$ commute
with each other. The index can only change when $\boldsymbol{\lambda}$
crosses the Clifford spectrum, and is thus constant on the connected
components of the Clifford resolvent. For a discussion of how this
index might serve as a sort of $K$-theory charge of a D-brane, see
\cite{berenstein2012matrix_embeddings}. 

A simple example, following \cite{Kisil_monogenic_func_calc}, will
illuminate this phenomena. For $d=3$ the choice we make for the three
$\gamma_{j}$ is to use the Pauli spin matrices, 
\begin{equation*}
\gamma_{1}=\sigma_{x}=\left[\begin{array}{cc}
0 & 1\\
1 & 0
\end{array}\right],\ \gamma_{2}=\sigma_{y}=\left[\begin{array}{cc}
0 & -i\\
i & 0
\end{array}\right],\ \gamma_{3}=\sigma_{z}=\left[\begin{array}{cc}
1 & 0\\
0 & -1
\end{array}\right].
\end{equation*}
The easiest way to understand this example is to first examine symmetries.
If $U=[u_{ij}]\in O(d)$ has determinant one then 
\begin{equation*}
\hat{\gamma}_{j}=\sum_{r=1}^{d}u_{jr}\gamma_{r}
\end{equation*}
form not just a representation of the Clifford relations, but a representation
that is unitarily equivalent to the original representation. Thus
the $\Lambda(\gamma_{1},\gamma_{2},\gamma_{3})$ has rotational symmetry.
The proof of Lemma~\ref{lem:Unitary_action_on_spectrum} works for
points in the Clifford resolvent, and shows the following.

Suppose $\left(A_{1},\dots,A_{d}\right)$ is a $d$-tuple of Hermitian
matrices in \textbf{$\boldsymbol{M}_{n}(\mathbb{C})$} and that $U\in SO(d)$.
Suppose $\boldsymbol{\mathbb{\lambda}}\in\mathbb{R}^{d}$. The $d$
matrices
\begin{equation*}
\hat{A}_{j}=\sum_{s}u_{js}A_{s}
\end{equation*}
are also Hermitian, and if $\boldsymbol{\mathbb{\lambda}}\notin\Lambda\left(A_{1},\dots,A_{d}\right)$
then $U\boldsymbol{\mathbb{\lambda}}\notin\Lambda\left(\hat{A}_{1},\dots,\hat{A}_{d}\right)$
and 
\begin{equation*}
\mathrm{sig}\left(L_{\boldsymbol{\lambda}}(A_{1},\dots,A_{d})\right)=\mathrm{sig}\left(L_{U\boldsymbol{\lambda}}(\hat{A}_{1},\dots,\hat{A}_{d})\right),
\end{equation*}
again following from unitary equivalence. Thus we can improve our earlier theorem.
 
\begin{thm}
\label{thm:Symmetry_Cliff_and_signature}Suppose $\left(A_{1},\dots,A_{d}\right)$
is a $d$-tuple of Hermitian matrices in \textbf{$\boldsymbol{M}_{n}(\mathbb{C})$},
that $U\in SO(d)$. Let 
\begin{equation*}
\hat{A}_{j}=\sum_{s}u_{js}A_{s}.
\end{equation*}
If there exists $Q\in U(n)$ such that
\begin{equation*}
Q\hat{A_{j}}Q^{*}=A_{j}
\end{equation*}
for all $j$ then 
\begin{equation*}
\boldsymbol{\mathbb{\lambda}}\in\Lambda\left(A_{1},\dots,A_{d}\right)\iff U\boldsymbol{\mathbb{\lambda}}\in\Lambda\left(A_{1},\dots,A_{d}\right)
\end{equation*}
and if $\boldsymbol{\mathbb{\lambda}}\notin\Lambda\left(A_{1},\dots,A_{d}\right)$
then 
\begin{equation*}
\mathrm{sig}\left(L_{\boldsymbol{\lambda}}(A_{1},\dots,A_{d})\right)=\mathrm{sig}\left(L_{U\boldsymbol{\lambda}}(A_{1},\dots,A_{d})\right).
\end{equation*}
\end{thm}

\begin{thm}
The Clifford spectrum of $(\gamma_{1},\gamma_{2},\gamma_{3})$ is
the unit sphere. Moreover
\begin{equation*}
\mu_{\boldsymbol{\lambda}}^{C}(\gamma_{1},\gamma_{2},\gamma_{3})=\left|\left|\boldsymbol{\lambda}\right|-1\right|
\end{equation*}
and 
\begin{equation*}
\mathrm{sig}\left(L_{\boldsymbol{\lambda}}(\gamma_{1},\gamma_{2},\gamma_{3})\right)=\begin{cases}
2 & \text{if }\left|\boldsymbol{\lambda}\right|<1\\
0 & \text{if }\left|\boldsymbol{\lambda}\right|>1
\end{cases}.
\end{equation*}
\end{thm}

\begin{proof}
By Theorem~\ref{thm:Symmetry_Cliff_and_signature} we need only deal
with the special case $\boldsymbol{\lambda}=(0,0,z)$. We find that
\begin{align*}
 & L_{(0,0,z)}(\gamma_{1},\gamma_{2},\gamma_{3})\\
 & \quad=\sigma_{x}\otimes\sigma_{x}+\sigma_{y}\otimes\sigma_{y}+\left(\sigma_{z}-zI\right)\otimes\sigma_{z}\\
 & \quad=\left[\begin{array}{cccc}
0 & 0 & 0 & 1\\
0 & 0 & 1 & 0\\
0 & 1 & 0 & 0\\
1 & 0 & 0 & 0
\end{array}\right]+\left[\begin{array}{cccc}
0 & 0 & 0 & -1\\
0 & 0 & 1 & 0\\
0 & 1 & 0 & 0\\
-1 & 0 & 0 & 0
\end{array}\right]+\left[\begin{array}{cccc}
1 & 0 & 0 & 0\\
0 & -1 & 0 & 0\\
0 & 0 & -1 & 0\\
0 & 0 & 0 & 1
\end{array}\right]+\left[\begin{array}{cccc}
-z & 0 & 0 & 0\\
0 & -z & 0 & 0\\
0 & 0 & z & 0\\
0 & 0 & 0 & z
\end{array}\right]\\
 & \quad=\left[\begin{array}{cccc}
1-z & 0 & 0 & 0\\
0 & -1-z & 2 & 0\\
0 & 2 & -1+z & 0\\
0 & 0 & 0 & 1+z
\end{array}\right].
\end{align*}
This has spectrum
\begin{equation*}
\left\{ 1\pm z,-1\pm\sqrt{z^{2}+4}\right\} 
\end{equation*}
and the results follows.
\end{proof}

This result can be generalized in at least two ways. The Pauli spin
matrices can be replaced by generators $(X_{1},X_{2},X_{3})$ of a
fuzzy sphere. In that case \cite{berenstein2012matrix_embeddings},
the Clifford spectrum and signature come out the same, but as the
matrix size increases the norms of the commutators $\left[X_{j},X_{k}\right]$
decrease. We will generalize this is a different direction, calculating
$\Lambda(\gamma_{1},\dots,\gamma_{d})$ for higher $d$. 

To illuminate the proof of the general case we look at the effect
of conjugating $L_{(0,0,z)}(\gamma_{1},\gamma_{2},\gamma_{3})$ by
the unitary
\begin{equation}
\label{eqn:Q_diagonalizing}
Q=\frac{1}{\sqrt{2}}\left[\begin{array}{cccc}
0 & 1 & 1 & 0\\
1 & 0 & 0 & 1\\
1 & 0 & 0 & -1\\
0 & 1 & -1 & 0
\end{array}\right].
\end{equation}
We find
\begin{equation*}
Q\left(\sigma_{x}\otimes\sigma_{x}\right)Q^{*}=\left[\begin{array}{cccc}
1 & 0 & 0 & 0\\
0 & 1 & 0 & 0\\
0 & 0 & -1 & 0\\
0 & 0 & 0 & -1
\end{array}\right]
\end{equation*}
\begin{equation*}
Q\left(\sigma_{y}\otimes\sigma_{y}\right)Q^{*}=\left[\begin{array}{cccc}
1 & 0 & 0 & 0\\
0 & -1 & 0 & 0\\
0 & 0 & 1 & 0\\
0 & 0 & 0 & -1
\end{array}\right]
\end{equation*}
\begin{equation*}
Q\left(\sigma_{z}\otimes\sigma_{z}\right)Q^{*}=\left[\begin{array}{cccc}
-1 & 0 & 0 & 0\\
0 & 1 & 0 & 0\\
0 & 0 & 1 & 0\\
0 & 0 & 0 & -1
\end{array}\right]
\end{equation*}
\begin{equation*}
Q\left(-zI\otimes\sigma_{z}\right)Q^{*}=\left[\begin{array}{cccc}
0 & 0 & 0 & -z\\
0 & 0 & -z & 0\\
0 & -z & 0 & 0\\
-z & 0 & 0 & 0
\end{array}\right]
\end{equation*}
and so $QL_{(0,0,z)}(\gamma_{1},\gamma_{2},\gamma_{3})Q^{*}$ breaks
again into $2$-by-$2$ blocks, specifically
\begin{equation*}
\left[\begin{array}{cc}
1 & -z\\
-z & -3
\end{array}\right],\ \left[\begin{array}{cc}
1 & -z\\
-z & 1
\end{array}\right].
\end{equation*}
Notice that $Q$ is a matrix whose columns are joint approximate eigenvalues
of $\sigma_{x}\otimes\sigma_{x}$ and $\sigma_{y}\otimes\sigma_{y}$.
It also diagonalizes $\sigma_{z}\otimes\sigma_{z}$ since $\sigma_{z}=i\sigma_{y}\sigma_{x}$
and so 
\begin{equation*}
\sigma_{z}\otimes\sigma_{z}=-\left(\sigma_{y}\otimes\sigma_{y}\right)\left(\sigma_{x}\otimes\sigma_{x}\right).
\end{equation*}

\section{Conventions for representations of the Clifford relations}

The minimal representation of the Clifford relations is unique when
$d$ is even. When $d$ is odd we know $\gamma_{d}$ will be a scalar
multiple of a product of the other $\gamma_{j}$. We need to specify
which multiple to determine which of the two irreducible representations
we are using.

If $\gamma_{1},\dots,\gamma_{d}$ is our choice for an irreducible
representations for some odd $d$, we can get to a representation
$\alpha_{1},\dots,\alpha_{d+2}$ by using
\begin{align*}
\alpha_{j} & =\gamma_{j}\otimes\sigma_{x},\quad(j\leq d)\\
\alpha_{d+1} & =I\otimes\sigma_{y},\\
\alpha_{d+2} & =I\otimes\sigma_{z}.
\end{align*}
If we have
\begin{equation*}
\gamma_{d}=\epsilon_{d}\gamma_{d-1}\cdots\gamma_{2}\gamma_{1}
\end{equation*}
then 
\begin{align*}
\epsilon_{d+2}\alpha_{d+1}\cdots\alpha_{2}\alpha_{1} & =\epsilon_{d+2}\left(I\otimes\sigma_{y}\right)\left(\gamma_{d}\otimes\sigma_{x}\right)\left(\gamma_{d-1}\otimes\sigma_{x}\right)\dots\left(\gamma_{1}\otimes\sigma_{x}\right)\\
 & =\epsilon_{d+2}\left(I\otimes\sigma_{y}\right)\left(\gamma_{d}\otimes\sigma_{x}\right)\left(\left(\gamma_{d-1}\cdots\gamma_{1}\right)\otimes\sigma_{x}^{d-1}\right)\\
 & =\epsilon_{d+2}\left(\gamma_{d}\otimes\left(-i\sigma_{z}\right)\right)\left(\epsilon_{d}^{-1}\gamma_{d}\otimes I\right)\\
 & =-i\epsilon_{d+2}\epsilon_{d}^{-1}\left(I\otimes\sigma_{z}\right)\\
 & =-i\epsilon_{d+2}\epsilon_{d}^{-1}\alpha_{d+2}
\end{align*}
so we should set
\begin{equation*}
\epsilon_{d+2}=i\epsilon_{d}.
\end{equation*}
As a base convention, we are using $\gamma_{1}=\sigma_{x}$, $\gamma_{2}=\sigma_{y}$,
$\gamma_{3}=\sigma_{z}$ so $\gamma_{3}=i\gamma_{2}\gamma_{1}$ meaning
\begin{equation}
\label{eqn:Clifford_sign_convention}
\epsilon_{d}=i^{(d+3)/2}.
\end{equation}

\section{The Clifford spectrum of Clifford matrices \label{sec:Clifford_of_Clifford}}

The case where $d$ is even is rather boring, but we work out this
case part way as some of the results will be helpful in the odd case.

\begin{thm}
If $d$ is even and $\gamma_{1},\dots,\gamma_{d}$ is an irreducible
representation of the Clifford relations then
\begin{equation*}
\Lambda(\gamma_{1},\dots,\gamma_{d}) = \{ 0 \}.
\end{equation*}
\end{thm}

\begin{proof}
Let $G_{j}=\gamma_{j}\otimes\gamma_{j}$ so that
\begin{equation*}
L_{\boldsymbol{0}}=L_{\boldsymbol{0}}\left(\gamma_{1},\dots,\gamma_{d}\right)=\sum_{j=1}^d G_{j}.
\end{equation*}
We want to show this is singular. The $G_{j}$ are matrices of size
$2^{d}$ and they pairwise commute. Each $G_{j}$ is Hermitian and
squares to one, so the joint spectrum of $(G_{1},\dots,G_{d})$ is
a subset of $\{\pm1\}^{\times d}$. Since replacing a single $\gamma_{j}$
by $-\gamma_{j}$ leads to another representation of the Clifford
relations, unitarily equivalent to the original, there is a unitary
by which we can conjugate to fix all the $G_{j}$ except the one that
is negated. Similarly we can exchange a pair, $\gamma_{j}$ with $\gamma_{k}$,
and so get another unitary whose action swaps $G_{j}$ and $G_{k}$.
Thus the joint spectrum must have the same multiplicity at all points
of $\{\pm1\}^{\times d}$ and so that common multiplicity must be
one. That means there is an  orthonormal  basis 
\begin{equation}
\label{eqn:Basis_diagonalizing}
\left\{ \boldsymbol{b}_{\boldsymbol{p}}\,\mid\,\boldsymbol{p}\in\{\pm1\}^{\times d}\right\} 
\end{equation}
such that 
\begin{equation*}
G_{j}\boldsymbol{b}_{\boldsymbol{p}}=p_{j}\boldsymbol{b}_{\boldsymbol{p}}\quad(1\leq j\leq d).
\end{equation*}
Thus the spectrum of $L_{\boldsymbol{0}}$ consists only of even numbers,
including zero. An eigenvalue of $0$ will arise from every $\boldsymbol{p}$ that
has an equal number of $-1$ and $+1$ coordinates.

Notice that in the basis given in (\ref{eqn:Basis_diagonalizing}) is, up to order, the unique basis in which the localizer at zero becomes diagonal.  In the case $d=3$ these vectors are the columns of $Q$ in (\ref{eqn:Q_diagonalizing}).  For the general case we just need to know we have such a basis.

We want to show that there are no points in the Clifford spectrum besides $\boldsymbol{0}$.
By the discussion in Section~\ref{sec:symmetries} we know that Theorem~\ref{thm:symmetry_in_spectrum} applies.  Thus it suffices to consider $\boldsymbol{\lambda} $ equal to $ (x,0,\dots0)$.
 If we set $H= I \otimes \gamma_1$ then
\begin{equation*}
L_{(x,0,\dots0)}\left(\gamma_{1},\dots,\gamma_{d}\right)=\sum G_{j}-xH.
\end{equation*}
Notice that $G_{1}H=HG_{1}$ and $G_{j}H=-HG_{j}$ for $j\geq2$.
Thus
\begin{equation*}
G_{1}H\boldsymbol{b}_{\boldsymbol{p}}=HG_{1}\boldsymbol{b}_{\boldsymbol{p}}=H\left(p_{1}\boldsymbol{b}_{\boldsymbol{p}}\right)=p_{1}H\boldsymbol{b}_{\boldsymbol{p}}
\end{equation*}
and, for $j\geq2$, 
\begin{equation*}
G_{j}H\boldsymbol{b}_{\boldsymbol{p}}=-HG_{j}\boldsymbol{b}_{\boldsymbol{p}}=-H\left(p_{j}\boldsymbol{b}_{\boldsymbol{p}}\right)=-p_{j}H\boldsymbol{b}_{\boldsymbol{p}}.
\end{equation*}
The joint eigenspaces are one dimensional so we know there is a unit
scalar $\theta_{\boldsymbol{p}}$ so that
\begin{equation*}
H\boldsymbol{b}_{\boldsymbol{p}}=\theta_{\boldsymbol{p}}\boldsymbol{b}_{\tilde{\boldsymbol{p}}}
\end{equation*}
 where
\begin{equation*}
\tilde{\boldsymbol{p}}=(p_{1},-p_{2},-p_{3},\dots,-p_{d}).
\end{equation*}
Since $H^2=I$ we find that 
\begin{equation*}
    \boldsymbol{b}_{\boldsymbol{p}}=H^2\boldsymbol{b}_{\boldsymbol{p}}=H(\theta_{\boldsymbol{p}}\boldsymbol{b}_{\tilde{\boldsymbol{p}}})=\theta_{\boldsymbol{p}}\theta_{\tilde{\boldsymbol{p}}}\boldsymbol{b}_{\boldsymbol{p}}
\end{equation*}
thus proving that $\overline{\theta_{\boldsymbol{p}}}=\theta_{\tilde{\boldsymbol{p}}}$.
Thus, in this basis, we know that $L_{(x,0,\dots0)}\left(\gamma_{1},\dots,\gamma_{d}\right)$
decomposes into $2$-by-$2$ blocks, one block for each pair $\left\{ \boldsymbol{p},\tilde{\boldsymbol{p}}\right\} $.
To avoid a double count we assume 
\begin{equation*}
s=\sum_{j=2}^{n}p_{j}\geq1
\end{equation*}
(notice $s$ must be odd).  The corresponding block is 
\begin{equation*}
\left[\begin{array}{cc}
p_{1}+s & -\overline{\theta_{\boldsymbol{p}}}x\\
-\theta_{\boldsymbol{p}}x & p_{1}-s
\end{array}\right].
\end{equation*}
This block has determinant
\begin{equation*}
-x^{2}+\left(p_{1}^{2}-s^{2}\right)=-x^{2}+\left(1-s^{2}\right).
\end{equation*}
Since $1-s^{2}\leq0$
this block can be singular only when $s=1$ and $x=0$.
\end{proof}

When $d$ is odd we can build on what we learned in the even case.  The exception is when $d=1$.  To get a zero sphere we would need a generator of the full Clifford algebra $\mathbb{C}\oplus\mathbb{C}$, an Hermitian matrix with spectrum $\{-1,1\}$.    For larger $d$ we need to use an irreducible representation of the Clifford algebra to be able to see the $K$-theory.

\begin{thm}
If $d \geq 3$ is odd and $\gamma_{1},\dots,\gamma_{d}$ is an irreducible
representation of the Clifford relations then
\begin{equation*}
\Lambda(\gamma_{1},\dots,\gamma_{d}) = S^{d-1},
\end{equation*}
the unit sphere in $\mathbb{R}^d$.
Moreover
\begin{equation*}
\mathrm{sig}\left(L_{\boldsymbol{\lambda}}(\gamma_{1},\dots,\gamma_{d})\right)
=
\left(\begin{array}{c}
d-1\\
 \tfrac{1}{2}(d-1) 
\end{array}\right)
\end{equation*}
if $0\leq|\boldsymbol{\lambda}|<1$ and this signature is
zero for $|\boldsymbol{\lambda}|>1$.
\end{thm}

\begin{proof}
In this case, we can assume $\gamma_1,\dots ,\gamma_{d-1}$ satisfy the conditions of the last theorem, so we use the orthonormal basis 
\begin{equation*}
\left\{ \boldsymbol{b}_{\boldsymbol{p}}\,\mid\,\boldsymbol{p}\in\{\pm1\}^{\times (d-1)}\right\} 
\end{equation*}
such that 
\begin{equation*}G_{j}\boldsymbol{b}_{\boldsymbol{p}}=p_{j}\boldsymbol{b}_{\boldsymbol{p}}\quad(1\leq j\leq d-1).
\end{equation*}
We have also to consider $G_d = \gamma_d \otimes \gamma_d$.  According to (\ref{eqn:Clifford_sign_convention}) we will have
\begin{equation*}
G_{d}= (-1)^{\frac{d-1}{2}} \prod_{j<d}G_j.
\end{equation*}
This means that 
\begin{equation*}
L_{\boldsymbol{0}}=L_{\boldsymbol{0}}\left(\gamma_{1},\dots,\gamma_{d}\right)=\sum_{j=1}^d G_{j}
\end{equation*}
has still $\boldsymbol{b}_{\boldsymbol{p}}$ as an eigenvector, but now with eigenvalue
\begin{equation*}
\alpha(\boldsymbol{p}) =  \sum_{j=1}^{d-1} p_j + (-1)^{\frac{d-1}{2}} \prod_{j=1}^{d-1} p_j.
\end{equation*}

These eigenvalues are now odd so $L_{\boldsymbol{0}}$ is now nonsingular.
We again have rotational symmetry so we can restrict our attention to $\boldsymbol{\lambda}$ of the form $(x,0,\dots,0)$.
The localizer
\begin{equation*}
L_{(x,0,\dots0)}\left(\gamma_{1},\dots,\gamma_{d}\right)
\end{equation*}
again breaks into $2$-by-$2$ blocks, where the blocks are \begin{equation*}
B_{\boldsymbol{p}} = \left[\begin{array}{cc}
 \alpha(\boldsymbol{p}) & -\overline{\theta_{p}}x\\
-\theta_{p}x &  \alpha(\tilde{\boldsymbol{p}})
\end{array}\right].
\end{equation*}
 where
\begin{equation*}
\tilde{\boldsymbol{p}}=(p_{1},-p_{2},-p_{3},\dots,-p_{d-1}).
\end{equation*}
This block has determinant
$\alpha(\boldsymbol{p})\alpha(\tilde{\boldsymbol{p}}) - x^2$.

We need to figure out the positive values of $\alpha(\boldsymbol{p})\alpha(\tilde{\boldsymbol{p}})$ at $\boldsymbol{p}$. We first compute
\begin{align*}
\alpha(\tilde{\boldsymbol{p}}) & =2p_{1}-\sum_{j=1}^{d-1}p_{j}-(-1)^{\frac{d-1}{2}}\prod_{j=1}^{d-1}p_{j}\\
 & =2p_{1}-\alpha(\boldsymbol{p}).
\end{align*}
The only way to avoid having $\alpha(\boldsymbol{p})$ and $\alpha(\tilde{\boldsymbol{p}})$ of opposite signs is to have $\alpha(\boldsymbol{p})=p_1=\pm1$.
When this happens,
the block $B_{\boldsymbol{p}}$ has   eigenvalues $1\pm x$. Thus $L_{(x,0,\dots0)}\left(\gamma_{1},\dots,\gamma_{d}\right)$ is singular only when $x=\pm1$, so the Clifford spectrum is (by symmetry) the unit sphere.

As to the signature, we need only compute this at $\boldsymbol{\lambda}=\boldsymbol{0}$.
We first look for all solutions to $|\alpha(\boldsymbol{p})|=1$.
If there are $k$ occurrences of $-1$ in $p_{1},\dots,p_{d-1}$ then
\begin{equation*}
\alpha(\boldsymbol{p})=(d-1)-2k+(-1)^{\frac{d-1}{2}}(-1)^{k}.
\end{equation*}
The product works out as $\pm1$ so  $k$ must be a small range,
specifically
\begin{equation*}
\tfrac{1}{2}(d-1)-1\leq k\leq\tfrac{1}{2}(d-1)+1.
\end{equation*}
When $k=\tfrac{1}{2}(d-1)$ we find 
\begin{align*}
(d-1)-2k & =0\\
(-1)^{\frac{d-1}{2}}(-1)^{k} & =1
\end{align*}
so we have a solution with $\alpha(\boldsymbol{p})=1$. When $k=\frac{1}{2}(d+1)$
we find 
\begin{align*}
(d-1)-2k & =-2\\
(-1)^{\frac{d-1}{2}}(-1)^{k} & =-1
\end{align*}
which we do not have a solution. When $k=\frac{1}{2}(d-3)$ we find
\begin{align*}
(d-1)-2k & =2\\
(-1)^{\frac{d-1}{2}}(-1)^{k} & =-1
\end{align*}
and so have a solution with $\alpha(\boldsymbol{p})=1$. There are
no solutions with $\alpha(\boldsymbol{p})=-1$.

If we add the condition that $p_1=\alpha(\textbf{p})=1$ we find that $\boldsymbol{p}\mapsto\tilde{\boldsymbol{p}}$ 
swaps the two types of solutions.  Thus we need only count the solutions
with $k=\frac{1}{2}(d-1)$.
Let us count the solutions with $p_{1}=1$ and $k=\frac{1}{2}(d-1)$.
Here we need to pick the places for $-1$ out
of $d-2$ places, so there are
\begin{equation*}
{d-2 \choose \frac{1}{2}(d-1)}
\end{equation*}
values for $\boldsymbol{p}$ here. This is also the number of blocks with non-zero
signature. 
Each has block has index $2$ so the overall signature is
\begin{equation*}
2{d-2 \choose \frac{1}{2}(d-1)}={d-1 \choose \frac{1}{2}(d-1)}.
\end{equation*}
\end{proof}

Software to numerically verify the value of the signature, for $d$ up to about $11$, is available on GitHub \cite{software_repo}.

\section{Spectrum and index of a 4D system \label{sec:4D_system}}

There is still much we do not know about almost commuting matrices, and some of what we do know is  
nonconstructive.  We know that Lin's theorem \cite{LinAlmostCommutingHermitian1994} states that for every pair of almost commuting Hermitian matrices there is close-by a pair of commuting Hermitian matrices. 
The proofs of this theorem are sufficiently non-constructive that we do not have any reasonable algorithm to find the nearby commuting pair.  For three Hermitian matrices, the result fails; the essential example that shows this is nice and constructive, as it is just the three matrices generating a fuzzy sphere \cite{HastingsLoringWannier}.  For real symmetric matrices, we have a real version of Lin's theorem \cite{LorSorensenDisk}, but we do not know if the real version of Lin's theorem holds for three matrices. As we show in this section, we pick up a $K$-theoretical obstruction for five real matrices, so we know that the real version fails in this case.   There is an example \cite[\S 6]{BoersLorRuiz_Pictures_K-theory} that shows this obstruction in nontrivial, but the prior example in that study is far from constructive.

There are abstract results \cite[\S 4-5]{BoersLorRuiz_Pictures_K-theory} related to bivariant $K$-theory for real $C^*$-algebras that tell us there are five 
real symmetric matrices of $(A_1,\dots,A_5)$ norm one with
\begin{equation}
    \label{eqn:fuzzy_sphere_commutators}
    \| [ A_j , A_k ] \|
\end{equation}
and
\begin{equation}
    \label{eqn:fuzzy_sphere_square_sum}
    \left\| \sum_j A_j^2 -I \right\|
\end{equation}
arbitrarily small, with the property that these are not close to five commuting real symmetric matrices.  The obstruction is $K$-theoretical and can be expressed as the fact that
\begin{equation*}
    L_{\boldsymbol{0}} (A_1,\dots,A_5)
\end{equation*}
has nontrivial signature.  From this we can conclude that
\begin{equation*}
   \Lambda (A_1,\dots,A_5)
\end{equation*}
is very similar to a four-sphere.  We can say it is a compact set that separates the origin from infinity and that it is close to the unit sphere.  We cannot rule out a union of concentric spheres, for example.  More critically, it is just about impossible to unwind the homotopy arguments used in $E$-theory so we have no way of writing down these matrices.

A possible construction of such matrices is to utilize
the theoretical models of 4D topological insulators that have a real Hamiltonian, from \cite[\S 3]{schnyder2008}, and then truncate these models to be finite.  By the results of \cite{LoringSchuba_even_dim_localizer} we know that if we use very large models, these will have the correct $K$-theory and have the quantities in (\ref{eqn:fuzzy_sphere_commutators}) and (\ref{eqn:fuzzy_sphere_square_sum}) as small as desired.  Potentially, this might only work with matrices so large we cannot even store them on a computer.

In this section we will show that this is a somewhat practical approach, in that the $K$-theory comes out to be as expected in small models.  To get commutators that are truly small, say less than $1/100$, might require a large computer, but should be possible if one is curious enough.  We cannot prove that the $K$-theory stays nontrivial for all models larger than the one we work with, but past experience indicates that this should be the case.

To demonstrate the appearance of higher dimensional spheres in the Clifford spectrum of physical systems, we consider a 4D tight-binding lattice that has been previously realized in an electric circuit \cite{price_four-dimensional_2020,wang_circuit_2020}. This model has four sites per unit cell, on which the annihilation operators can be labelled as $a_{m,n,j,l}$, $b_{m,n,j,l}$, $c_{m,n,j,l}$, $d_{m,n,j,l}$, where the indices specify the unit cell in all four dimensions. The corresponding creation operators are given by the conjugate transpose of these operators (again denoted using physics notation), $a^\dagger, b^\dagger, c^\dagger, d^\dagger$. The lattice's Hamiltonian in the standard tight-binding basis can be divided into four sets of terms, the on-site energies
\begin{equation}
    H_{\textrm{on}} = M \sum_{m,n,j,l} \left(a_{m,n,j,l}^\dagger a_{m,n,j,l} + b_{m,n,j,l}^\dagger b_{m,n,j,l} - c_{m,n,j,l}^\dagger c_{m,n,j,l}  - d_{m,n,j,l}^\dagger d_{m,n,j,l}\right),
\end{equation}
the nearest neighbor couplings within a single unit cell
\begin{equation}
    H_{\textrm{NN,in}} = -t \sum_{m,n,j,l} \left(c_{m,n,j,l}^\dagger a_{m,n,j,l} - b_{m,n,j,l}^\dagger d_{m,n,j,l} + d_{m,n,j,l}^\dagger a_{m,n,j,l} + b_{m,n,j,l}^\dagger c_{m,n,j,l}\right) + \textrm{H.c.} ,
\end{equation}
the nearest neighbor couplings between different unit cells
\begin{align}
    H_{\textrm{NN,out}} = -t \sum_{m,n,j,l} ( & a_{m+1,n+1,j,l}^\dagger c_{m,n,j,l} - d_{m+1,n+1,j,l}^\dagger b_{m,n,j,l} + c_{m-1,n,j,l}^\dagger a_{m,n,j,l} \notag \\ 
    &- b_{m-1,n,j,l}^\dagger d_{m,n,j,l} + a_{m,n,j+1,l+1}^\dagger d_{m,n,j,l} + c_{m,n,j+1,l+1}^\dagger b_{m,n,j,l} \notag  \\ 
    & + d_{m,n,j-1,l}^\dagger a_{m,n,j,l} + b_{m,n,j-1,l}^\dagger c_{m,n,j,l}) + \textrm{H.c.},
\end{align}
and the longer-range couplings
\begin{align}
    H_{\textrm{LR}} = -t_1 \sum_{m,n,j,l} ( & a_{m+1,n+1,j+1,l+1}^\dagger a_{m,n,j,l} + b_{m+1,n+1,j+1,l+1}^\dagger b_{m,n,j,l} \notag \\
    &- c_{m+1,n+1,j+1,l+1}^\dagger c_{m,n,j,l} - d_{m+1,n+1,j+1,l+1}^\dagger d_{m,n,j,l})  + \textrm{H.c.}.
\end{align}
Then, the full lattice Hamiltonian is given by
\begin{equation}
    H = H_{\textrm{on}} + H_{\textrm{NN,in}} + H_{\textrm{NN,out}} + H_{\textrm{LR}}. \label{eq:4DHtot}
\end{equation}
In these equations, the notation ``$+ \textrm{H.c.}$'' is used to indicate that every term's Hermitian conjugate is also included, e.g., if $t a_{m,n,j,l}^\dagger b_{m',n',j',l'}$ is included, then so is $\bar{t} b_{m',n',j',l'}^\dagger a_{m,n,j,l}$.

In the tight-binding basis, this 4D lattice's position operators $X_{1,2,3,4}$ are simply diagonal matrices, in which each diagonal element $[X_j]_{kk}$ is the position in the $j$th dimension of the $k$th unit cell. For simplicity for the examples considered in Figs.\ \ref{fig:AI} and \ref{fig:A}, all four lattice sites in each unit cell are assigned the same spatial coordinate, and unit cells are separated by the same lattice constant $a$ in all four spatial directions. Thus, we can form the 4D tight-binding lattice's spectral localizer as
\begin{equation}
    L_{\boldsymbol{\lambda}=(x_1,x_2,x_3,x_4,E)}(X_1,X_2,X_3,X_4,H)=\sum_{j=1}^4 \kappa \left(X_{j} - \lambda_j\right)\otimes \Gamma_j + \left(H - \lambda_5\right) \otimes \Gamma_5. \label{eq:4DL}
\end{equation}
Here, we are using the $d=5$ representation of the Clifford relations defined as
\begin{gather*}
\Gamma_1 = \left[\begin{array}{cccc}
0 & 0 & 1 & 0\\
0 & 0 & 0 & -1\\
1 & 0 & 0 & 0\\
0 & -1 & 0 & 0
\end{array}\right], \;\;\;\;
\Gamma_2 = \left[\begin{array}{cccc}
0 & 0 & -i & 0\\
0 & 0 & 0 & -i\\
i & 0 & 0 & 0\\
0 & i & 0 & 0
\end{array}\right], \;\;\;\;
\Gamma_3 = \left[\begin{array}{cccc}
0 & 0 & 0 & 1\\
0 & 0 & 1 & 0\\
0 & 1 & 0 & 0\\
1 & 0 & 0 & 0
\end{array}\right] \\
\Gamma_4 = \left[\begin{array}{cccc}
0 & 0 & 0 & -i\\
0 & 0 & i & 0\\
0 & -i & 0 & 0\\
i & 0 & 0 & 0
\end{array}\right], \;\;\;\;
\Gamma_5 = \left[\begin{array}{cccc}
1 & 0 & 0 & 0\\
0 & 1 & 0 & 0\\
0 & 0 & -1 & 0\\
0 & 0 & 0 & -1
\end{array}\right].
\end{gather*}
Finally, we note that the scaling parameter $\kappa$ in Eq.\ (\ref{eq:4DL}) also serves to ensure that the spectral localizer has consistent units, and thus $\kappa$ has units of energy/distance.

\begin{figure*}
\includegraphics{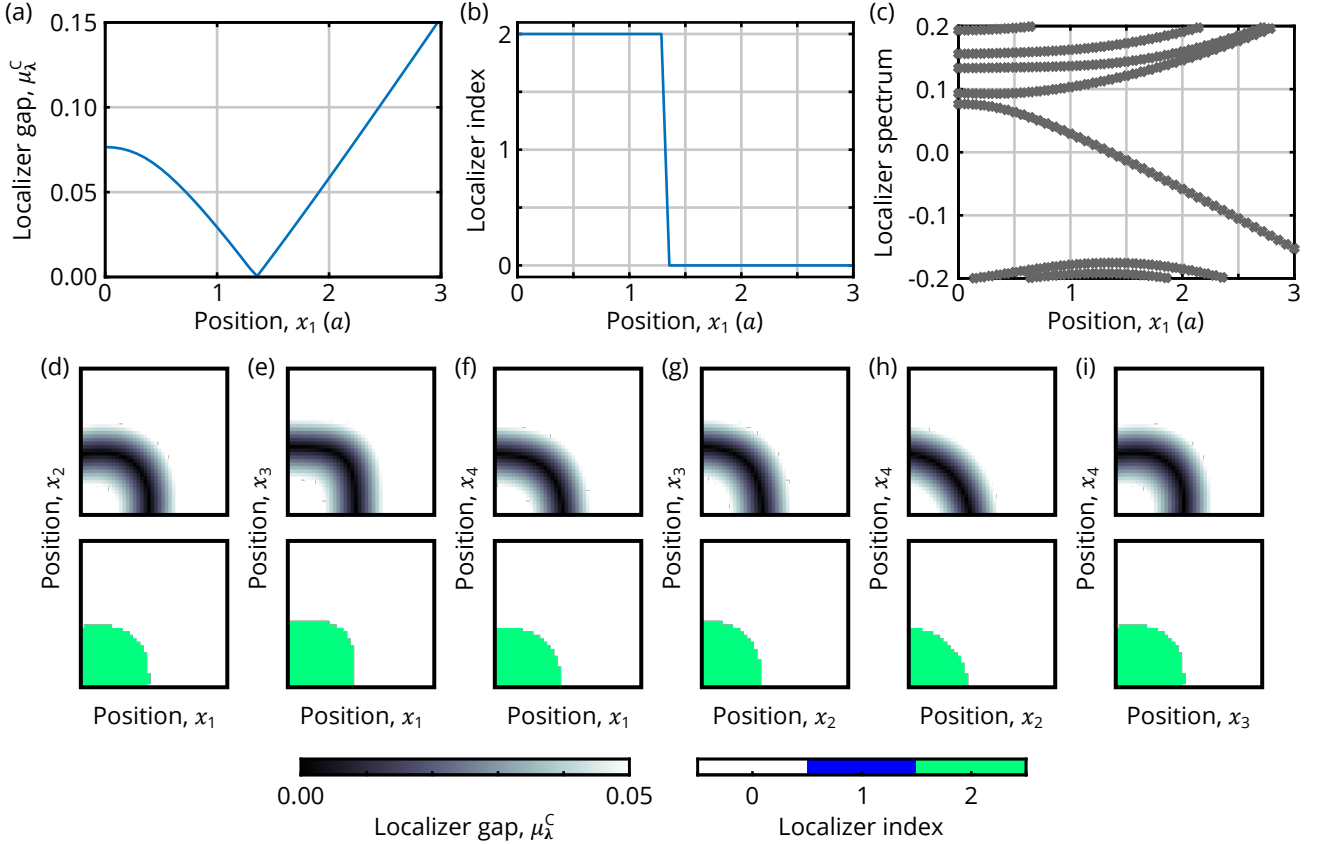}
\caption{(a) Localizer gap $\mu_{\boldsymbol{\lambda}}^{\textrm{C}}$ for a 5-by-5-by-5-by-5 4D class AI lattice whose Hamiltonian is given by Eq.\ (\ref{eq:4DHtot}), with $t_1 = 0.8 t$ and $m = t/2$. Here, only $x_1$ is varied, with $x_2 = x_3 = x_4 = 0$, and $\kappa = 0.1(t/a)$, where $a$ is the lattice constant. (b) Localizer index for the same system. (c) Spectral flow of the 20 eigenvalues of $L_{\boldsymbol{\lambda}}$ closest to zero over the same position variation. (d)-(i) Slices of the localizer gap (top) and localizer index (bottom) over the surfaces $x_i$ and $x_j$, with $i,j = 1,2,3,4$. Any coordinate not shown in a given plot is fixed to zero. The positions shown are varied between $0$ and $3a$. The lattice is centered at the origin.}
\label{fig:AI}
\end{figure*}

For $M,t,t_1 \in \mathbb{R}$, this lattice is in class AI of the Altland-Zirnbauer classification \cite{schnyder2008,kitaev2009,ryu2010topological}, and possesses bosonic time-reversal symmetry $\mathcal{T}^2 = +1$, but no other local symmetries. Direct calculation of the system's Clifford spectrum reveals that it is a 4-dimensional surface that is approximately $S^{4}$ (i.e., a 4D sphereoid), see Fig.\ \ref{fig:AI}d-i. Moreover, the $K$-theory element associated with this lattice's Clifford spectrum is a local marker equivalent for the second Chern number, and can be calculated as
\begin{equation*}
    \textrm{ind}_{\boldsymbol{\lambda}}(X_1,X_2,X_3,X_4,H) = \tfrac{1}{2}\textrm{sig}(L_{\boldsymbol{\lambda}}(X_1,X_2,X_3,X_4,H)). \label{eq:ind4D}
\end{equation*}
For class AI, the time-reversal symmetry fixes this index to always be even. As can be seen in Fig.\ \ref{fig:AI}, inside of the closed 4D surface that forms the Clifford spectrum, the index is seen to be $2$, while outside the index is $0$.


Even with the small system size, the commutators are substantially smaller than in the Clifford matrix examples in Section~\ref{sec:Clifford_of_Clifford}.  We need to normalize things, as none of these matrices have norm one. Numerical estimates tell us
\begin{equation*}
    \left\Vert \left[X_{j},H\right]\right\Vert /(\left\Vert X_{j}\right\Vert \left\Vert H\right\Vert )
\end{equation*}
are approximately equal to $0.29$ for odd $j$ and $0.21$ for even $j$ and that each $X_j$ has norm $2.0$ and finally that $H$ has norm close to $4$.  In order to get non-trivial $K$-theory we have to rescale these, replacing $X_j$ by $(0.1)X_j$. 
For much larger systems size, we have a theorem \cite{LoringSchuba_even_dim_localizer} 
to tell us that the index of the localizer at the origin will equal the second Chern number.  There is substantial numerical evidence that this equality will hold for much smaller system sizes. See, for example,  \cite{loring2015Pseudospectra_topo_ins,loring2019bulk_spectrum_quasicrystal,lozano_schuba2019Half-signature}.

\begin{figure*}
\includegraphics{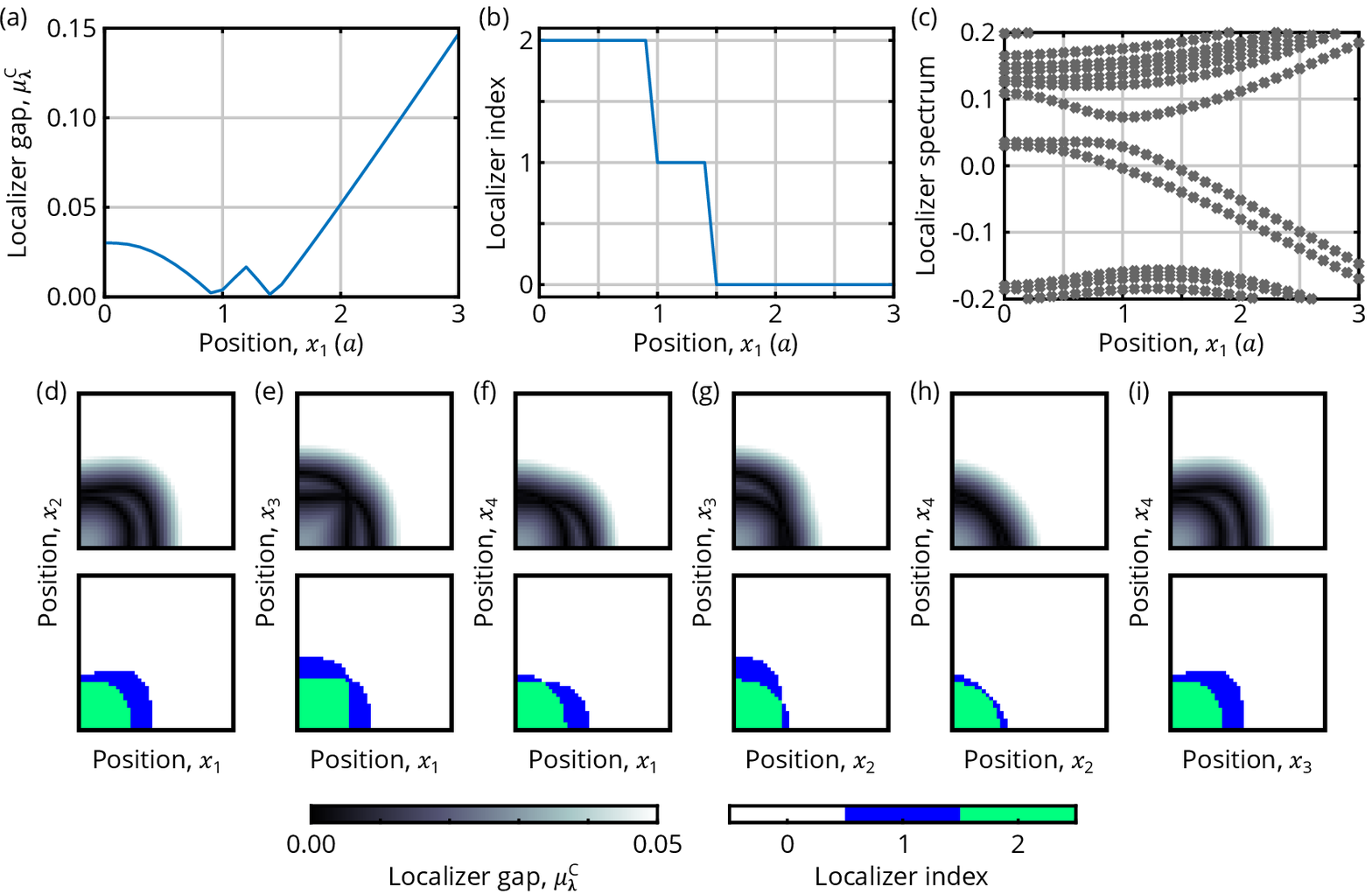}
\caption{(a) Localizer gap $\mu_{\boldsymbol{\lambda}}^{\textrm{C}}$ for a 5-by-5-by-5-by-5 4D class A lattice whose Hamiltonian is given by Eq.\ (\ref{eq:4DHtot}), with $t_1 = 0.8 e^{0.1 i \pi} t$ and $M = t/2$. Here, only $x_1$ is varied, with $x_2 = x_3 = x_4 = 0$, and $\kappa = 0.1(t/a)$, where $a$ is the lattice constant. (b) Localizer index for the same system. (c) Spectral flow of the 20 eigenvalues of $L_{\boldsymbol{\lambda}}$ closest to zero over the same position variation. (d)-(i) Slices of the localizer gap (top) and localizer index (bottom) over the surfaces $x_i$ and $x_j$, with $i,j = 1,2,3,4$. Any coordinate not shown in a given plot is fixed to zero. The positions shown are varied between $0$ and $3a$. The lattice is centered at the origin.}
\label{fig:A}
\end{figure*}

By relaxing the reality of the coupling coefficients and choosing $t_1 \in \mathbb{C}$, the lattice's time-reversal symmetry is broken and it instead falls into class A of the Altland-Zirnbauer classification, which represents systems with no local symmetries. As such, its topological index determined by Eq.\ (\ref{eq:ind4D}) can now be any integer, and odd values of this index are seen in Fig.\ \ref{fig:A}. For this system, the Clifford spectrum is two intersecting 4D spheroids, that merge together as time-reversal symmetry is restored. For choices of $\boldsymbol{\lambda}$ in the interior of both spheroids, the index remains $2$, while the index outside of both surfaces is $0$. However, in between the two surfaces, a region that was inaccessible in the time-reversal symmetric system, the index is $1$.

A system with width of $5$ lattice units is too small to try the methods of \cite{LoringSchuba_even_dim_localizer} to create a ``fuzzy sphere'' with (\ref{eqn:fuzzy_sphere_commutators}) and (\ref{eqn:fuzzy_sphere_square_sum}) both small and
all the matrices normalized to have norm one.  To do this properly, we would want a round sample in place of a square sample, and we would use a larger system.  By keeping the position observables to be of norm one as the system size grows, we would get smaller and smaller commutators.

Instead, we have found strong numerical evidence that the Clifford spectrum for the observables in our small system looks somewhat like a four-sphere.  In particular, it seems that a ray leaving the origin in any direction while staying at energy zero will cross the Clifford spectrum just once for $t_1 \in \mathbb{R}$ (Fig.\ \ref{fig:AI}), or once or twice for $t_1 \in \mathbb{C}$ (Fig.\ \ref{fig:A}). The hope is that someone will see a pattern here, find a method to produce fuzzy four-spheres based on real-symmetric matrices, and prove that the Clifford spectrum of those matrices is a four-sphere.

All of the algorithms necessary to reproduce these numerical results are available on GitHub \cite{software_repo}.

\section{Open problems and future directions}

In the typical application in physics the matrices involved are $D$ matrices $X_j$ that specify position and the Hamiltonian $H$.  In this case, the spectral localizer at zero
is comprised of $H\otimes e_{D+1}$ and $\sum X_j \otimes e_j$.  In the continuum the latter is the Fourier transform of a standard Dirac operator.  The term  involving the Hamiltonian can be seen as a perturbation of the Dirac operator.  In this way, one can deduce the structure of the  spectrum of this perturbed Dirac operator \cite{franca2023obstructions,schulz2022semimetals_ELP,schuba2022Localizer_semimetals}.  

There are, however, many other incompatible observables to which we have want to apply the localizer, such as momentum operators and current operators.  We would like to have the basic theory of 
the Clifford spectrum built up to better support this ongoing research in mathematical physics.  To this end, we present some conjectures.  Some of these may require techniques from geometry or other areas of mathematics outside of operator theory.

There are patterns emerging in the types of spaces that show up as the Clifford spectrum in examples.  For the most part, these patterns have not been explained.   The following conjecture is due to Kisil \cite{Kisil_monogenic_func_calc}.  Kisil sketched a possible proof, but the conjecture seems to be still open. 

\begin{conjecture}
The Clifford spectrum $\Lambda(A_{1},\dots,A_{d})$ is always nonempty, given $A_{1},\dots,A_{d}$ Hermitian matrices in $\boldsymbol{M}_{n}(\mathbb{C})$.
\end{conjecture}

This is true for $d=2$, and in the commutative case, since in each case the Clifford spectrum is equal to a standard form of spectrum \cite{loring2015Pseudospectra_topo_ins}.  

In all the examples found so far, the Clifford spectrum of $d$ matrices in $\boldsymbol{M}_{n}(\mathbb{C})$ either has cardinality $n$ or less or contains a two-dimensional space.  

\begin{conjecture}
If the Clifford spectrum $\Lambda(A_{1},\dots,A_{d})$ for Hermitian matrices in $\boldsymbol{M}_{n}(\mathbb{C})$ is finite then it has cardinality at most $n$.
\end{conjecture}

\begin{conjecture}
The Clifford spectrum $\Lambda(A_{1},\dots,A_{d})$ of Hermitian matrices in $\boldsymbol{M}_{n}(\mathbb{C})$ is never a one-manifold.
\end{conjecture}

While we have not seen an example where the Clifford spectrum is a  non-orientable manifold, we believe such an example can exist.   Indeed, there may be matrix models based on recent physical experiments \cite{Li_acoustic_moebius2022,wang2023Moebius_optical} that lead to this.

We end with a more open-ended challenge.

\begin{problem}
Define a fuzzy four-sphere, similar to the usual fuzzy two-sphere, that can be generated by five almost commuting real-symmetric matrices, with arbitrarily small commutators possible.  This should have Clifford spectrum equal to the standard unit four-sphere and at points inside this sphere the
localizer should have nonzero signature.
\end{problem}

\rule[0.5ex]{1\columnwidth}{1pt}

\section*{Acknowledgements}

A.C.\ and T.L.\ acknowledge support from the Laboratory Directed Research and Development program at Sandia National Laboratories. T.L.\ acknowledges support from the National Science Foundation, grant DMS-2110398. This work was performed, in part, at the Center for Integrated Nanotechnologies, an Office of Science User Facility operated for the U.S. Department of Energy (DOE) Office of Science. Sandia National Laboratories is a multimission laboratory managed and operated by National Technology \& Engineering Solutions of Sandia, LLC, a wholly owned subsidiary of Honeywell International, Inc., for the U.S.\ DOE's National Nuclear Security Administration under contract DE-NA-0003525. The views expressed in the article do not necessarily represent the views of the U.S.\ DOE or the United States Government.


\begin{thebibliography}{10}

\bibitem{Arveson_Dirac_operator}
William Arveson.
\newblock The {D}irac operator of a commuting {$d$}-tuple.
\newblock {\em J. Funct. Anal.}, 189(1):53--79, 2002.

\bibitem{berenstein2012matrix_embeddings}
David Berenstein and Eric Dzienkowski.
\newblock Matrix embeddings on flat {$\mathbb{R}^3$} and the geometry of
  membranes.
\newblock {\em Physical Review D}, 86(8):086001, 2012.

\bibitem{BoersLorRuiz_Pictures_K-theory}
Jeffrey~L. Boersema, Terry~A. Loring, and Efren Ruiz.
\newblock Pictures of {$KK$}-theory for real {$C^*$}-algebras and almost
  commuting matrices.
\newblock {\em Banach J. Math. Anal.}, 10(1):27--47, 2016.

\bibitem{software_repo}
Alexander Cerjan.
\newblock Repository for code used in this manusctipt.
\newblock \url{https://github.com/acerjan/spectral_localizer_4d_class_ai},
  September 2023.

\bibitem{CerjanLoring2022TopoPhtonics}
Alexander Cerjan and Terry~A Loring.
\newblock An operator-based approach to topological photonics.
\newblock {\em Nanophotonics}, 11(21):4765--4780, 2022.

\bibitem{cerjan_quadratic_2022}
Alexander Cerjan, Terry~A. Loring, and Fredy Vides.
\newblock Quadratic pseudospectrum for identifying localized states.
\newblock {\em J. Math. Phys.}, 64(2):023501, 2023.

\bibitem{cheng2023revealing}
Wenting Cheng, Alexander Cerjan, Ssu-Ying Chen, Emil Prodan, Terry~A. Loring,
  and Camelia Prodan.
\newblock Revealing topology in metals using experimental protocols inspired by
  {$K$}-theory.
\newblock {\em Nat. Commun.}, to appear.

\bibitem{deBadyn_Karczmarek2015emergent_geometry}
Mathias~Hudoba de~Badyn, Joanna~L Karczmarek, Philippe Sabella-Garnier, and Ken
  Huai-Che Yeh.
\newblock Emergent geometry of membranes.
\newblock {\em Journal of High Energy Physics}, 2015(11):1--33, 2015.

\bibitem{DeBonisLorSver_joint_spectrum}
Patrick~H. DeBonis, Terry~A. Loring, and Roman Sverdlov.
\newblock Surfaces and hypersurfaces as the joint spectrum of matrices.
\newblock {\em Rocky Mountain J. Math.}, 52(4):1319--1343, 2022.

\bibitem{dixon2023PhotonicHeterostrcutre}
Kahlil~Y Dixon, Terry~A Loring, and Alexander Cerjan.
\newblock Classifying topology in photonic heterostructures with gapless
  environments.
\newblock {\em arXiv preprint arXiv:2303.17135}, 2023.

\bibitem{Doll_schuba_Z2_flows_skew_localizer}
Nora Doll and Hermann Schulz-Baldes.
\newblock Skew localizer and {$\Bbb{Z}_2$}-flows for real index pairings.
\newblock {\em Adv. Math.}, 392:Paper No. 108038, 42, 2021.

\bibitem{franca2023obstructions}
Selma Franca and Adolfo~G. Grushin.
\newblock Obstructions in trivial metals as topological insulator zero-modes.
\newblock {\em arXiv preprint arXiv:2304.01983}, 2023.

\bibitem{haldane_model_1988}
F.~D.~M. Haldane.
\newblock Model for a {Quantum} {Hall} {Effect} without {Landau} {Levels}:
  {Condensed}-{Matter} {Realization} of the ``{Parity} {Anomaly}''.
\newblock {\em Phys. Rev. Lett.}, 61(18):2015--2018, October 1988.

\bibitem{HastingsLoringWannier}
Matthew~B. Hastings and Terry~A. Loring.
\newblock Almost commuting matrices, localized {W}annier functions, and the
  quantum {H}all effect.
\newblock {\em J. Math. Phys.}, 51(1):015214, 2010.

\bibitem{Jefferies_McIn_Weyl_Calc}
Brian Jefferies and Alan McIntosh.
\newblock The {W}eyl calculus and {C}lifford analysis.
\newblock {\em Bull. Austral. Math. Soc.}, 57(2):329--341, 1998.

\bibitem{Kisil_monogenic_func_calc}
Vladimir~V. Kisil.
\newblock M\"{o}bius transformations and monogenic functional calculus.
\newblock {\em Electron. Res. Announc. Amer. Math. Soc.}, 2(1):26--33, 1996.

\bibitem{KisilRamirex_Cliff_func_calc}
Vladimir~V. Kisil and Enrique Ram\'{\i}rez~de Arellano.
\newblock The {R}iesz-{C}lifford functional calculus for non-commuting
  operators and quantum field theory.
\newblock {\em Math. Methods Appl. Sci.}, 19(8):593--605, 1996.

\bibitem{kitaev2009}
Alexei Kitaev.
\newblock Periodic table for topological insulators and superconductors.
\newblock {\em AIP Conference Proceedings}, 1134(1):22--30, 2009.

\bibitem{Li_acoustic_moebius2022}
Tianzi Li, Juan Du, Qicheng Zhang, Yitong Li, Xiying Fan, Fan Zhang, and
  Chunyin Qiu.
\newblock Acoustic m\"obius insulators from projective symmetry.
\newblock {\em Phys. Rev. Lett.}, 128:116803, Mar 2022.

\bibitem{LinAlmostCommutingHermitian1994}
Huaxin Lin.
\newblock Almost commuting selfadjoint matrices and applications.
\newblock In {\em Operator algebras and their applications ({W}aterloo, {ON},
  1994/1995)}, volume~13 of {\em Fields Inst. Commun.}, pages 193--233. Amer.
  Math. Soc., Providence, RI, 1997.

\bibitem{liu2018KitaevChainsJosephJunct}
Dillon~T Liu, Javad Shabani, and Aditi Mitra.
\newblock Long-range kitaev chains via planar josephson junctions.
\newblock {\em Physical Review B}, 97(23):235114, 2018.

\bibitem{loring2015Pseudospectra_topo_ins}
Terry~A Loring.
\newblock K-theory and pseudospectra for topological insulators.
\newblock {\em Annals of Physics}, 356:383--416, 2015.

\bibitem{loring2019bulk_spectrum_quasicrystal}
Terry~A Loring.
\newblock Bulk spectrum and {$K$}-theory for infinite-area topological
  quasicrystals.
\newblock {\em Journal of Mathematical Physics}, 60(8):081903, 2019.

\bibitem{LoringSchuba_even_dim_localizer}
Terry~A. Loring and Hermann Schulz-Baldes.
\newblock The spectral localizer for even index pairings.
\newblock {\em J. Noncommut. Geom.}, 14(1):1--23, 2020.

\bibitem{LorSorensenDisk}
Terry~A. Loring and Adam P.~W. S{\o}rensen.
\newblock Almost commuting self-adjoint matrices: {T}he real and self-dual
  cases.
\newblock {\em Rev. Math. Phys.}, 28(7):1650017, 39, 2016.

\bibitem{lozano_schuba2019Half-signature}
Edgar Lozano~Viesca, Jonas Schober, and Hermann Schulz-Baldes.
\newblock Chern numbers as half-signature of the spectral localizer.
\newblock {\em Journal of Mathematical Physics}, 60(7):072101, 2019.

\bibitem{price_four-dimensional_2020}
Hannah~M. Price.
\newblock Four-dimensional topological lattices through connectivity.
\newblock {\em Phys. Rev. B}, 101(20):205141, May 2020.

\bibitem{ryu2010topological}
Shinsei Ryu, Andreas~P Schnyder, Akira Furusaki, and Andreas W~W Ludwig.
\newblock Topological insulators and superconductors: tenfold way and
  dimensional hierarchy.
\newblock {\em New Journal of Physics}, 12(6):065010, 2010.

\bibitem{schnyder2008}
Andreas~P. Schnyder, Shinsei Ryu, Akira Furusaki, and Andreas W.~W. Ludwig.
\newblock Classification of topological insulators and superconductors in three
  spatial dimensions.
\newblock {\em Phys. Rev. B}, 78:195125, Nov 2008.

\bibitem{schulz2022semimetals_ELP}
Hermann Schulz-Baldes and Tom Stoiber.
\newblock Invariants of disordered semimetals via the spectral localizer.
\newblock {\em Europhysics Letters}, 136(2):27001, 2022.

\bibitem{schuba2022Localizer_semimetals}
Hermann Schulz-Baldes and Tom Stoiber.
\newblock Spectral localization for semimetals and {C}allias operators.
\newblock {\em arXiv preprint arXiv:2203.15014}, 2022.

\bibitem{sykora2016fuzzy_space_kit}
Andreas Sykora.
\newblock The fuzzy space construction kit.
\newblock {\em arXiv preprint arXiv:1610.01504}, 2016.

\bibitem{wang2023Moebius_optical}
Jiawei Wang, Sreeramulu Valligatla, Yin Yin, Lukas Schwarz, Mariana
  Medina-S{\'a}nchez, Stefan Baunack, Ching~Hua Lee, Ronny Thomale, Shilong Li,
  Vladimir~M Fomin, et~al.
\newblock Experimental observation of berry phases in optical m{\"o}bius-strip
  microcavities.
\newblock {\em Nature Photonics}, 17(1):120--125, 2023.

\bibitem{wang_circuit_2020}
You Wang, Hannah~M. Price, Baile Zhang, and Y.~D. Chong.
\newblock Circuit implementation of a four-dimensional topological insulator.
\newblock {\em Nat. Commun.}, 11(1):2356, May 2020.

\end{thebibliography}
\end{document}